\documentclass[11pt,a4paper]{article}
\usepackage{amsfonts,amsgen,amstext,amsbsy,amsopn,amsfonts,amssymb,amscd}
\usepackage[leqno]{amsmath}
\usepackage[amsmath,amsthm,thmmarks]{ntheorem}
\usepackage{epsf,epsfig}
\usepackage{float}
\usepackage{dsfont}
\usepackage{ebezier,eepic}
\usepackage{color}
\usepackage{tikz}
\usepackage{multirow}
\usepackage{mathrsfs}
\usepackage{graphicx}
\usepackage{subfigure}
\newtheorem{thm}{Theorem}[section]

\newtheorem{prop}[thm]{Proposition}

\newtheorem{lem}[thm]{Lemma}
\newtheorem{example}[thm]{Example}
\newtheorem{false statement}{False statement}
\newtheorem{cor}[thm]{Corollary}

\theoremstyle{definition}

\newtheorem{conj}[thm]{Conjecture}

\makeatletter \@addtoreset{equation}{section}

\def\hh{\mathcal{H}}

\def\hl{\mathcal{L}}

\def\hf{\mathcal{F}}
\def\hg{\mathcal{G}}

\def\ha{\mathcal{A}}
\def\hb{\mathcal{B}}
\def\hs{\mathcal{S}}

\begin{document}

\title{\bf\Large A product version of the Hilton-Milner Theorem II}
\date{}
\author{Peter Frankl\footnote{R\'{e}nyi Institute, Budapest, Hungary. Email: \texttt{frankl.peter@renyi.hu}.}\quad\quad
Jian Wang \footnote{Department of Mathematics, Sichuan University, Chengdu, 610065, China.  Email: \texttt{wangjianmath01@scu.edu.cn}.}
}

\maketitle

\begin{abstract}
Two families $\mathcal{F},\mathcal{G}$ of $k$-subsets of $\{1,2,\ldots,n\}$ are called {\it non-trivial cross-intersecting} if $F\cap G\neq \emptyset$ for all $F\in \mathcal{F}, G\in \mathcal{G}$ and  $\cap \{F\colon F\in \mathcal{F}\}=\emptyset=\cap \{G\colon G\in\mathcal{G}\}$. In this note, we establish the product version of the  Hilton-Milner Theorem for $k\geq 8$ in the full range. That is,  if $\mathcal{F},\mathcal{G}\subset \binom{[n]}{k}$ are non-trivial cross-intersecting, $n\geq 2k+1$ and $k\geq 8$, then
\[
|\mathcal{F}||\mathcal{G}|\leq  \left(\binom{n-1}{k-1}- \binom{n-k-1}{k-1} +1\right)^2.
\]
\end{abstract}

\section{Introduction}

Let $[n]$ be the standard $n$-set $\{1,2,\ldots,n\}$ and $2^{[n]}$ its power set. Subsets of $2^{[n]}$ are called {\it families}.

For $0\leq k\leq n$ set $\binom{[n]}{k}=\{F\subset [n]\colon |F|=k\}$. A subset $\hf$ of $\binom{[n]}{k}$ is called a {\it $k$-uniform} family or simply a {\it $k$-graph}.

A family $\hf\subset 2^{[n]}$ is called {\it intersecting} if $F\cap F'\neq \emptyset$ for all $F,F'\in \hf$. One of the oldest and most fundamental results in extremal set theory is the Erd\H{o}s-Ko-Rado Theorem.

\begin{thm}[\cite{EKR}]
Suppose that $\hf\subset \binom{[n]}{k}$ is intersecting  and $n\geq 2k\geq 2$. Then
\begin{align}\label{ineq-ekr}
|\hf| \leq \binom{n-1}{k-1}.
\end{align}
\end{thm}

The {\it full star} $\hs=\{S\in \binom{[n]}{k}\colon 1\in S\}$ shows that \eqref{ineq-ekr} is best possible. Hilton and Milner proved that for $n>2k$ excluding stars leads to considerably better bounds.

Define the {\it Hilton-Milner family}
\[
\hh(n,k):=\left\{H\in \binom{[n]}{k}\colon 1\in H,\ H\cap \{2,3,\ldots,k+1\}\neq \emptyset\right\} \cup \{\{2,3,\ldots,k+1\}\}.
\]
For $n>2k$,
\[
|\hh(n,k)| =\binom{n-1}{k-1}-\binom{n-k-1}{k-1}+1<\binom{n-1}{k-1}.
\]

Let us say that $\hf$ is {\it non-trivial} if $\cap \{F\colon F\in \hf\}=\emptyset$.

\begin{thm}[\cite{HM67}]
Let $n> 2k\geq 4$. Suppose that $\hf\subset \binom{[n]}{k}$ is intersecting and non-trivial. Then
\begin{align}\label{ineq-nontrival}
|\hf| \leq |\hh(n,k)|.
\end{align}
Moreover, equality holds in \eqref{ineq-nontrival} only if $\hf$ is isomorphic to $\hh(n,k)$.
\end{thm}

Being a strong stability result, the Hilton-Milner Theorem is widely used in extremal set theory. During the years various new proofs were given (cf.  \cite{FFuredi},\cite{Mors},\cite{FT},\cite{KZ},\cite{HK},\cite{F2017}).

Let us say that two families $\hf,\hg$ are {\it cross-intersecting} (CI for short) if $F\cap G\neq \emptyset$ for all pairs $F\in \hf$ and $G\in \hg$.

This is a very natural extension of the notion ``intersecting". As a matter of fact it naturally arises in the proof of results concerning intersecting families. It plays a central role already in the paper of Hilton and Milner.

Let us recall the first product theorem for CI families.

\begin{thm}[\cite{Pyber86}]
Suppose that $\hf,\hg\subset \binom{[n]}{k}$ are CI. Then
\begin{align}\label{ineq-pyber}
|\hf||\hg| \leq \binom{n-1}{k-1}^2.
\end{align}
\end{thm}

Let us note that in the case $n=2k$, $\binom{n}{k}=2\binom{n-1}{k-1}$ and $|\hf|+|\hg|\leq 2\binom{n-1}{k-1}$ is easy to prove. This implies \eqref{ineq-pyber} by the inequality between AM and GM. Hence in the sequel we always assume $n\geq 2k+1$.

Define $h(n,k):=|\hh(n,k)|$.
Based on our earlier work let us state a conjecture.

\begin{conj}
Suppose that $\hf,\hg\subset \binom{[n]}{k}$ are CI and both are non-trivial. Then
\begin{align}\label{ineq-conj}
|\hf||\hg| \leq h(n,k)^2.
\end{align}
\end{conj}

Needless to say that \eqref{ineq-conj} implies \eqref{ineq-nontrival} just as \eqref{ineq-pyber} implies \eqref{ineq-ekr}. Let us recall our recent result:

\begin{thm}[\cite{FW23}]
Suppose that $\hf,\hg\subset \binom{[n]}{k}$ are CI and both are non-trivial. Then \eqref{ineq-conj} holds for $n\geq 4k$ and $k\geq 8$.
\end{thm}

Let us also mention that the $k=3$ case of \eqref{ineq-conj} was settled in \cite{F26}.

The aim of the present paper is to prove \eqref{ineq-conj} for $k\geq 8$ in the range $4k\geq n\geq 2k+1$ thereby completely settling the case $k\geq 8$ of \eqref{ineq-conj}.

\begin{thm}\label{thm-main2}
Let  $4k\geq n\geq 2k+1$, $k\geq 8$. Suppose that  $\hf,\hg\subset \binom{[n]}{k}$ are CI and both are non-trivial. Then
\begin{align}\label{newineq-hfhg2}
|\hf||\hg|\leq  \left(\binom{n-1}{k-1}- \binom{n-k-1}{k-1} +1\right)^2=h(n,k)^2.
\end{align}
\end{thm}

Obviously equality holds in \eqref{newineq-hfhg2} if  $\hf=\hg=\hh(n,k)$. However this is not the only case.

\begin{example}
Let $F_0 , G_0 \in \binom{[2,n]}{k}$  with  $F_0 \cap G_0\neq \emptyset$.
 Define
 \begin{align*}
 \hf &=\{F_0\} \cup \left\{ F \in \binom{[n]}{k}\colon 1 \in F,\ F \cap G_0 \neq \emptyset \right\},\\[3pt]
 \hg&=\{G_0\} \cup \left\{  G \in \binom{[n]}{k}\colon  1 \in G,\ G \cap F_0 \neq \emptyset \right\}.
 \end{align*}
  Even that both families are isomorphic to the Hilton-Milner family, the isomorphism class of the pair $(\hf,\hg)$ depends on $|F_0 \cap G_0|$.
 \end{example}

The present note is complementing our earlier work that treated the case $n\geq 4k$.

How to use the inequality $2k<n\leq 4k$ to our advantage? The main point is that in this range $h(n,k)$ and  $\binom{n-1}{k-1}$ are close to each other. Let us formulate it into a proposition.

\begin{prop}
Let $2k<n\leq tk$. Then
\begin{align}\label{ineq-key0}
\frac{h(n,k)^2}{\binom{n-1}{k-1}^2}>1-2\left(\frac{2t-3}{2t-1}\right)^{k-1}.
\end{align}
\end{prop}

\begin{proof}
Since for any $0\leq z<1$, $(1-z)^2>1-2z$ it is sufficient to prove
$\frac{h(n,k)}{\binom{n-1}{k-1}}>1-\left(\frac{2t-3}{2t-1}\right)^{k-1}$.
Using $h(n,k)=\binom{n-1}{k-1}-\binom{n-k-1}{k-1}+1$, proving $\frac{\binom{n-k-1}{k-1}}{\binom{n-1}{k-1}}\leq \left(\frac{2t-3}{2t-1}\right)^{k-1}$
is sufficient. Noting $\frac{(n-k-j)(n-k-(k-j))}{(n-j)(n-(k-j))}\leq \left(\frac{n-\frac{3k}{2}}{n-\frac{k}{2}}\right)^2$, by $n\leq tk$ we have
\[
\frac{\binom{n-k-1}{k-1}}{\binom{n-1}{k-1}} =\prod_{1\leq i\leq k-1}\frac{n-k-i}{n-i} \leq \left(\frac{n-\frac{3}{2}k}{n-\frac{k}{2}}\right)^{k-1}
\leq \left(\frac{2t-3}{2t-1}\right)^{k-1}.
\]
\end{proof}

\section{Preliminaries}

There is an important partial order $\prec$, called {\it shifting order} for sets of equal size. For convenience let $(a_1,\ldots,a_k)$ denote the $k$-set $\{a_1,\ldots,a_k\}$ if $a_1<\ldots<a_k$. We define
\[
(a_1,\ldots,a_k)\prec (b_1,\ldots,b_k) \mbox{ iff } a_i\leq b_i \mbox{ holds for all }1\leq i\leq k.
\]

A $k$-graph $\hf\subset \binom{[n]}{k}$ is called {\it shifted} (or {\it initial}) if $(b_1,\ldots,b_k)\in \hf$ and $(a_1,\ldots,a_k)\prec (b_1,\ldots,b_k)$ always imply $(a_1,\ldots,a_k)\in \hf$.

Shifted $k$-graphs have many nice properties and their use can be traced back to the paper of Erd\H{o}s, Ko and Rado \cite{EKR}.

We also need the notion of {\it shift-resistant pair}. For a pair of families $\hf,\hg\subset \binom{[n]}{k}$, define the quantity
\[
w(\hf,\hg) =\sum_{F\in \hf}\sum_{i\in F} i + \sum_{G\in \hg}\sum_{j\in G} j.
\]
Let us fix  $\hf,\hg\subset \binom{[n]}{k}$ so that $\hf,\hg$ are CI and both are non-trivial, $|\hf||\hg|$ is maximal, moreover among such pairs $w(\hf,\hg)$ is minimal. If in such pair  $\hf$ and $\hg$ are not both shifted then we say that $(\hf,\hg)$ forms a {\it shift-resistant pair}. Clearly, in proving Theorem \ref{thm-main2} we may assume that either $\hf,\hg$ are both shifted or $(\hf,\hg)$ forms a  shift-resistant pair.

For the proof of Theorem \ref{thm-main2} we need the following.

\begin{prop}[\cite{FW23}]\label{prop-2.1}
Suppose that $\hf,\hg\subset \binom{[n]}{k}$ are CI and both are non-trivial, $|\hf||\hg|$ is maximal, moreover among such pairs $w(\hf,\hg)$ is minimal.  If $\hf,\hg$ are both shifted, then either $\min\{|\hf|,|\hg|\}\leq \binom{n}{k-2}$ or $|\hf||\hg|\leq h(n,k)^2$ holds. If $(\hf,\hg)$ forms a  shift-resistant pair, then $\min\{|\hf|,|\hg|\}\leq \binom{n-2}{k-2}+\binom{n-4}{k-2}$.
\end{prop}

There is also an important total order $<_L$, the {\it lexicographic order} defined on $F,G\in \binom{[n]}{k}$ by,
\[
F<_L G \mbox{ iff } \min\{i\colon i\in F\setminus G\}<\min\{i\colon i\in G\setminus F\}.
\]
For $1\leq m\leq \binom{n}{k}$ let $\hl(n,k,m)$ denote the $k$-graph whose edges are the first $m$ sets from $\binom{[n]}{k}$ in the lexicographic order.

Based on the Kruskal-Katona Theorem, Hilton proved the following very useful statement.

\begin{lem}[\cite{Hilton}]
Suppose that $\hf,\hg\subset \binom{[n]}{k}$ are CI. Then $\hl(n,k,|\hf|)$ and $\hl(n,k,|\hg|)$ are CI as well.
\end{lem}

Let us recall some standard notation. For $i\in [n]$ and $\hf\subset 2^{[n]}$, define
\[
\hf(\hat{i}) =\{F\in\hf\colon i\in F\}\mbox{ and } \hf(\bar{i}) =\{F\in\hf\colon i\notin F\}.
\]

Let us fix $n>2k\geq 4$ and for a family $\hh\subset \binom{[n]}{k}$ define $\hh_1=\hl(n,k,|\hh(\hat{1})|)$ and $\hh_0=\hl([2,n],k,|\hh(\bar{1})|)$ where $[2,n]$ indicates that we only consider $k$-subsets of $[2,n]$.

\begin{lem}\label{lem-2.3}
Suppose that $\hf,\hg\subset \binom{[n]}{k}$ are non-trivial and CI and assume $|\hf(\hat{1})|\geq \binom{n-2}{k-2}$, $|\hg(\hat{1})|\geq \binom{n-2}{k-2}$. Then
\begin{itemize}
  \item[(i)] $\hf_0\cup \hg_1$ and $\hf_1\cup \hg_0$ are intersecting.
  \item[(ii)] If $\min\{|\hf(\hat{1})|, |\hg(\hat{1})|\}> \binom{n-2}{k-2}$ then $\hf_0\cup \hg_1$ and $\hf_1\cup \hg_0$ are non-trivial as well.
\end{itemize}
\end{lem}

\begin{proof}
By CI property and Hilton's Lemma $(\hf_0,\hg_1)$ and $(\hf_1,\hg_0)$ are both CI. Also $\hf_1$ and $\hg_1$ are both stars. To prove (i) we need to show that $\hf_0$ and $\hg_0$ are also intersecting.

The first $\binom{n-2}{k-2}$ sets in the lexcographic order on $\binom{[n]}{k}$ are the supersets of $\{1,2\}$. Hence any $k$-subset of $[2,n]$ intersecting each of them must contain 2. This shows that both $\hf_0$ and $\hg_0$ are contained in the full star of 2 whence intersecting.

In case (ii) both $\hf_1$ and $\hg_1$ must contain some $k$-set not containing 2. As $\hf_0\neq \emptyset\neq\hg_0$, the non-triviality follows.
\end{proof}

\begin{cor}\label{cor-key}
Suppose that $\hf,\hg\subset \binom{[n]}{k}$ are non-trivial and CI. If $\min\{|\hf(\hat{1})|,|\hg(\hat{1})|\}>\binom{n-2}{k-2}$ then
\begin{align}\label{ineq-2.1}
|\hf||\hg|\leq h(n,k)^2.
\end{align}
\end{cor}

\begin{proof}
In view of Lemma \ref{lem-2.3}, both $\hf_0\cup \hg_1$ and $\hf_1\cup \hg_0$ are non-trivial and intersecting. By \eqref{ineq-nontrival},
\[
|\hf_0|+|\hg_1|\leq h(n,k), \ |\hf_1|+|\hg_0|\leq h(n,k).
\]
Adding these two inequalities: $|\hf|+|\hg|\leq 2h(n,k)$, implying \eqref{ineq-2.1}.
\end{proof}

For $\hf\subset \binom{[n]}{k}$,  define the {\it diversity} $\gamma(\hf)$ as $\min_{1\leq i\leq n}|\hf(\bar{i})|$.  We need  the diversity result for cross-intersecting families due to the first author and Kupavskii.

\begin{thm}[\cite{FK21}]\label{thm-fk}
Let $n\geq 2k$. Suppose that $\ha,\hb\subset \binom{n}{k}$ are cross-intersecting. If
\[
|\ha|,|\hb|> \binom{n-1}{k-1}-\binom{n-u-1}{k-1}+\binom{n-u-1}{k-u} \mbox{ for some  } u \mbox{ with } 3\leq u\leq k,
\]
then
\begin{align}\label{ineq-diversity2}
\gamma(\ha),\gamma(\hb) <\binom{n-u-1}{k-u},
\end{align}
moreover, both families have the same (unique) element of the largest degree.
\end{thm}

Let us recall two results from our earlier paper.

\begin{prop}[\cite{FW23}]\label{prop-2.3}
In proving Theorem \ref{thm-main2} we may assume that
\begin{align}
&\min \left\{|\hf|,|\hg|\right\}> \binom{n-3}{k-3}+\binom{n-4}{k-3}.
\end{align}
\end{prop}

\begin{prop}[\cite{FW23}]\label{prop-2.4}
Suppose that $\hf,\hg\subset \binom{[n]}{k}$ are cross-intersecting,  $n\geq 2k$ and
$\min\{|\hf|,|\hg|\}\geq \binom{n-3}{k-3}+\binom{n-4}{k-3}$. Then
\begin{align}
|\hf|+|\hg|\leq 2\binom{n-1}{k-1}.
\end{align}
\end{prop}

We also need the following inequality.

\begin{lem}
For $k\geq 8$ and $2k+1\leq n\leq 3k$,
\begin{align}\label{ineq-key}
\binom{n-2}{k-2}+2\binom{n-3}{k-2} +\frac{h(n,k)^2}{\binom{n-2}{k-2}+2\binom{n-3}{k-2}}> 2\binom{n-1}{k-1}.
\end{align}
\end{lem}

\begin{proof}
Note that
\[
\binom{n-2}{k-2}+2\binom{n-3}{k-2}=\left(\frac{k-1}{n-1}
+\frac{2(k-1)(n-k)}{(n-1)(n-2)}\right)\binom{n-1}{k-1}= \frac{(k-1)(3n-2k-2)}{(n-1)(n-2)}\binom{n-1}{k-1}.
\]
Set $x=\frac{(k-1)(3n-2k-2)}{(n-1)(n-2)}$.
Dividing both sides by $\binom{n-1}{k-1}$, we see that \eqref{ineq-key} is equivalent to
\[
x+  \frac{h(n,k)^2/\binom{n-1}{k-1}^{2}}{x}> 2.
\]
Since $n\leq 3k$, by \eqref{ineq-key0} we have
\[
\frac{h(n,k)^2}{\binom{n-1}{k-1}^2}>1-2\left(\frac{3}{5}\right)^{k-1}.
\]
Then it suffices to show that
\begin{align}\label{ineq-2.7}
1-2\left(\frac{3}{5}\right)^{k-1}> x(2-x).
\end{align}
Let $n=ck+1$ with $2\leq c \leq 3$. Then
\begin{align*}
1-x=1-\frac{(k-1)(3n-2k-2)}{(n-1)(n-2)}&=\frac{(n-1)(n-2)-(k-1)(3n-2k-2)}{(n-1)(n-2)}\\[3pt]
&=\frac{ck(ck-1)-(k-1)((3c-2)k+1)}{ck(ck-1)}\\[3pt]
&=\frac{(c-2)(c-1)k^2+(2c-3)k+1}{ck(ck-1)}\\[3pt]
&\geq \frac{k}{ck(ck-1)}\\[3pt]
&\geq \frac{1}{9k}.
\end{align*}
It follows that $x(2-x)=1-(1-x)^2\leq 1-\frac{1}{81k^2}$.
Thus to show \eqref{ineq-2.7} it suffices to show that
\begin{align}\label{ineq-2.8}
f(k):=162k^2\left(\frac{3}{5}\right)^{k-1}< 1.
\end{align}
It is easy to check that $f(24)\approx 0.737<1$. Moreover, for $k\geq 24$
\[
\frac{f(k+1)}{f(k)} = \frac{3}{5}\frac{(k+1)^2}{k^2}<\frac{3}{5}\cdot\frac{25^2}{24^2}<1.
\]
Thus \eqref{ineq-2.8} holds for $k\geq 24$.
For $k=8,9,\ldots,23$ and $2k+1\leq n\leq 3k$, one can check by direct computation that \eqref{ineq-key} holds.
\end{proof}
By almost the same proofs, one can also prove the following two inequalities. For self-containedness, we add their proofs in the Appendix.
\begin{lem}\label{lem-2.7}
For $k\geq 8$ and $2k+1\leq n\leq 4k$,
\begin{align}\label{ineq-key1}
\binom{n-2}{k-2}+\binom{n-4}{k-2} +\frac{h(n,k)^2}{\binom{n-2}{k-2}+\binom{n-4}{k-2}}> 2\binom{n-1}{k-1}.
\end{align}
\end{lem}

\begin{lem}\label{lem-2.8}
For $k\geq 8$ and $3k\leq n\leq 4k$,
\begin{align}\label{ineq-key2}
\binom{n}{k-2} +\frac{h(n,k)^2}{\binom{n}{k-2}}> 2\binom{n-1}{k-1}.
\end{align}
\end{lem}

\section{Proof of Theorem \ref{thm-main2}}

Let us first prove Theorem \ref{thm-main2} for  $3k\leq n\leq 4k$.

\begin{lem}\label{lem-main}
Suppose that $\hf,\hg\subset \binom{[n]}{k}$ are CI and both non-trivial, $4k\geq n\geq 3k$, $k\geq 8$. Then
\begin{align*}
|\hf||\hg|\leq h(n,k)^2.
\end{align*}
\end{lem}

\begin{proof}
Suppose that $|\hf||\hg|\geq h(n,k)^2$.
 By Propositions \ref{prop-2.3} and  \ref{prop-2.4} we have
\begin{align}\label{ineq-3.1}
|\hf|+|\hg| \leq 2\binom{n-1}{k-1}.
\end{align}

By symmetry assume $|\hf|\leq |\hg|$. Then by Proposition \ref{prop-2.1} we may assume either $|\hf|\leq \binom{n}{k-2}$ or $|\hf|\leq \binom{n-2}{k-2}+\binom{n-4}{k-2}$. Clearly $\binom{n-2}{k-2}+\binom{n-4}{k-2}<h(n,k)$. For $n\geq 3k$,
\begin{align*}
\frac{\binom{n-2}{k-2}+2\binom{n-3}{k-2}}{\binom{n}{k-2}}
&=\frac{(n-k+2)(n-k+1)}{n(n-1)}+\frac{2(n-k+2)(n-k+1)(n-k)}{n(n-1)(n-2)}\\[3pt]
&\geq \frac{(2k+2)(2k+1)}{3k(3k-1)}+\frac{2(2k+2)(2k+1)2k}{3k(3k-1)(3k-2)}\\[3pt]
&> \frac{4}{9}+\frac{16}{27}>1.
\end{align*}
Thus $\binom{n}{k-2}<\binom{n-2}{k-2}+2\binom{n-3}{k-2}\leq h(n,k)$.

If $|\hf|\leq \binom{n}{k-2}$, then for $3k\leq n\leq 4k$,
\[
|\hf|+|\hg| \geq |\hf|+ \frac{h(n,k)^2}{|\hf|}\geq \binom{n}{k-2} +\frac{h(n,k)^2}{\binom{n}{k-2}} \overset{\eqref{ineq-key2}}{>} 2\binom{n-1}{k-1},
\]
contradicting \eqref{ineq-3.1}.

If $|\hf|\leq \binom{n-2}{k-2}+\binom{n-4}{k-2}$, then for $2k+1\leq n\leq 4k$,
\[
|\hf|+|\hg| \geq |\hf|+ \frac{h(n,k)^2}{|\hf|}\geq \binom{n-2}{k-2}+\binom{n-4}{k-2} +\frac{h(n,k)^2}{\binom{n-2}{k-2}+\binom{n-4}{k-2}} \overset{\eqref{ineq-key1}}{>} 2\binom{n-1}{k-1},
\]
contradicting \eqref{ineq-3.1} again. Thus the lemma follows.
\end{proof}

\begin{proof}[Proof of Theorem \ref{thm-main2}]
Let $\hf,\hg\subset \binom{[n]}{k}$ be  CI
and both non-trivial with $|\hf||\hg|$ maximal. Suppose for contradiction that $|\hf||\hg|\geq h(n,k)^2$. By Lemma \ref{lem-main}, we may further assume $2k+1\leq n\leq 3k$.

By applying Theorem \ref{thm-fk} with $u=3$, we obtain that if
\begin{align*}
|\hf|,|\hg|\geq \binom{n-1}{k-1}-\binom{n-4}{k-1}+\binom{n-4}{k-3}&=\binom{n-2}{k-2}+\binom{n-3}{k-2}+\binom{n-4}{k-2}+\binom{n-4}{k-3}\\[3pt]
&=\binom{n-2}{k-2}+2\binom{n-3}{k-2},
\end{align*}
then $\gamma(\hf),\gamma(\hg) <\binom{n-4}{k-3}$. Moreover, both families have the same (unique) element of the largest degree, say 1. It follows that $|\hf(1)|,|\hg(1)|\geq \binom{n-2}{k-2}+2\binom{n-3}{k-2}- \binom{n-4}{k-3}>\binom{n-2}{k-2}$. Then by Corollary \ref{cor-key} we are done. Thus by symmetry we may assume that
\[
|\hf|< \binom{n-2}{k-2}+2\binom{n-3}{k-2}\leq h(n,k)\leq |\hg|.
\] Then
\[
|\hf|+|\hg| \geq |\hf|+ \frac{h(n,k)^2}{|\hf|}\geq \binom{n-2}{k-2}+2\binom{n-3}{k-2} +\frac{h(n,k)^2}{\binom{n-2}{k-2}+2\binom{n-3}{k-2}} \overset{\eqref{ineq-key}}{>} 2\binom{n-1}{k-1}.
\]
However, by Propositions \ref{prop-2.3} and  \ref{prop-2.4} we have
\[
|\hf|+|\hg| \leq 2\binom{n-1}{k-1},
\]
a contradiction.
\end{proof}

\section{Concluding remarks}

 In the present paper, except for the cases $4\leq k\leq 7$, we concluded the determination of $\max |\hf||\hg|$ for
 non-trivial CI $k$-graphs. However this result begs the question, what happens if $\hf$ and $\hg$ have different uniformities.
 Let $n>2k>2\ell\geq 4$ be integers and fix  $F_0\in \binom{[2,n]}{k}$, $G_0\in \binom{[2,n]}{\ell}$ with $F_0\cap G_0\neq \emptyset$.
  Define
 \begin{align*}
 \hf_0 &=\{F_0\} \cup \left\{ F \in \binom{[n]}{k}\colon 1 \in F,\ F \cap G_0 \neq \emptyset \right\},\\[3pt]
 \hg_0&=\{G_0\} \cup \left\{  G \in \binom{[n]}{\ell}\colon  1 \in G,\ G \cap F_0 \neq \emptyset \right\}.
 \end{align*}

 Let us state a conjecture:

 \begin{conj}
 Let $n\geq 2k>2\ell\geq 4$. Suppose that $\hf\subset \binom{[n]}{k}$, $\hg\subset \binom{[n]}{\ell}$ are non-trivial CI families. Then
 \[
 |\hf||\hg| \leq  |\hf_0||\hg_0|.
 \]
 \end{conj}
   Let us mention that the corresponding product version of the Erd\H{o}s-Ko-Rado Theorem was proved by Matsumoto and Tokushige (\cite{MT}) .

\vspace{8pt}
{\noindent \bf Acknowledgement.} The second author was supported by
National Natural Science Foundation of China Grant no. 12471316 and the Fundamental Research Funds for
 the Central Universities.

\begin{appendix}
\section{Proofs of Lemma \ref{lem-2.7} and Lemma \ref{lem-2.8}.}
\begin{proof}[Proof of Lemma \ref{lem-2.7}]
Note that
\[
\binom{n-2}{k-2}+\binom{n-4}{k-2}=\left(\frac{k-1}{n-1}
+\frac{(k-1)(n-k)(n-k-1)}{(n-1)(n-2)(n-3)}\right)\binom{n-1}{k-1}.
\]
Let $x=\frac{k-1}{n-1}
+\frac{(k-1)(n-k)(n-k-1)}{(n-1)(n-2)(n-3)}$.
Dividing both sides by  $\binom{n-1}{k-1}$, we see that \eqref{ineq-key1} is equivalent to
\[
x+  \frac{h(n,k)^2/\binom{n-1}{k-1}^{2}}{x}> 2.
\]
Since $n\leq 4k$, by \eqref{ineq-key0} we have
\[
\frac{h(n,k)^2}{\binom{n-1}{k-1}^2}>1-2\left(\frac{5}{7}\right)^{k-1},
\]
it suffices to show that
\begin{align}\label{ineq-A1}
1-2\left(\frac{5}{7}\right)^{k-1} > x(2-x).
\end{align}
Let $n=ck+1$ with $2\leq c \leq 4$. Then for $k\geq 8$ we have
\begin{align*}
1-x&=1-\frac{k-1}{n-1}
-\frac{(k-1)(n-k)(n-k-1)}{(n-1)(n-2)(n-3)}\\[3pt]
&=\frac{((c-1)k+1)((c^2-c+1)k^2-(2c+1)k+2)}{ck(ck-1)(ck-2)}\\[3pt]
&\geq \frac{(c-1)((c^2-c+1)k-2c-1))}{c^3k}\\[3pt]
&\geq \frac{(c-1)(c^2-c+1-\frac{2c+1}{8})}{c^3}\\[3pt]
&\geq \frac{(c-1)(8c^2-10c+7)}{8c^3}>\frac{1}{4}.
\end{align*}
It follows that $1-(1-x)^2=x(2-x)\leq 1-\frac{1}{16}$.
Thus, to show \eqref{ineq-A1} it suffices to show that $\left(\frac{5}{7}\right)^{k-1}< \frac{1}{32}$,
which is true for $k\geq 11$.
For $k=8,9,10$ and $2k+1\leq n\leq 4k$, one can check by direct computation that \eqref{ineq-key1} holds.
\end{proof}

\begin{proof}[Proof of Lemma \ref{lem-2.8}]
Note that
\[
\binom{n}{k-2}=\frac{(k-1)n}{(n-k+2)(n-k+1)}\binom{n-1}{k-1}.
\]
Let $x=\frac{(k-1)n}{(n-k+2)(n-k+1)}$.
Dividing both sides by  $\binom{n-1}{k-1}$, we see that \eqref{ineq-key2} is equivalent to
\[
x+  \frac{h(n,k)^2/\binom{n-1}{k-1}^{2}}{x}> 2.
\]
By $n\leq 4k$ and \eqref{ineq-key0},
it suffices to show that
\begin{align}\label{ineq-A2}
1-2\left(\frac{5}{7}\right)^{k-1} > x(2-x).
\end{align}
Let $n=ck$ with $3\leq c \leq 4$. Then
\begin{align*}
1-x=1-\frac{(k-1)n}{(n-k+2)(n-k+1)}
&=\frac{(c^2-3c+1)k^2+(4c-3)k+2}{((c-1)k+2)((c-1)k+1)}\\[3pt]
&= \frac{(c^2-3c+1)k^2+(4c-3)k+2}{(c-1)^2k^2+3(c-1)k+2}\\[3pt]
&\geq \frac{c^2-3c+1}{(c-1)^2}\\[3pt]
&= 1-\frac{c}{(c-1)^2}\geq \frac{1}{4}.
\end{align*}
It follows again that $x(2-x)\leq 1-\frac{1}{16}$.
Thus, to show \eqref{ineq-A2} it suffices to show that
$\left(\frac{5}{7}\right)^{k-1}< \frac{1}{32}$, which is true for $k\geq 11$.
For $k=8,9,10$ and $3k\leq n\leq 4k$, one can check by direct computation that \eqref{ineq-key2} holds.
\end{proof}

\end{appendix}

\end{document}